\documentclass[a4paper,11pt]{amsart}

\usepackage{amsthm}
\usepackage[utf8]{inputenc}
\usepackage[T1]{fontenc}
\usepackage{amsmath}
\usepackage{graphicx}
\usepackage{epsf, subfigure, verbatim}
\usepackage{epsfig}
\usepackage{amssymb}
\usepackage{amsfonts}
\usepackage{enumerate}
\usepackage{srcltx}

\theoremstyle{plain}
\newtheorem*{theorem*}{Theorem A}
\newtheorem{theorem}[equation]{Theorem}
\newtheorem{proposition}[equation]{Proposition}
\newtheorem{lemma}[equation]{Lemma}

\newtheorem{corollary}[equation]{Corollary}

\newtheorem{conjecture}[equation]{Conjecture}
\theoremstyle{remark}
\newtheorem{remark}[equation]{Remark}
\theoremstyle{definition}

\numberwithin{equation}{subsection}

\begin{document}

\title[Superlevel sets and nodal extrema of Laplace eigenfunctions]{Superlevel sets and nodal extrema of Laplace-Beltrami eigenfunctions}
\author{Guillaume Poliquin}
\let\thefootnote\relax
\thanks{Research supported by a NSERC scholarship}

\begin{abstract}
We estimate the volume of superlevel sets of Laplace-Beltra\-mi eigenfunctions on a compact Riemannian manifold. The proof uses the Green's function representation and the Bathtub principle. As an application, we obtain upper bounds on the distribution of the extrema of a Laplace-Beltrami eigenfunction over its nodal domains. Such bounds have been previously proved by L. Polterovich and M. Sodin in the case of compact surfaces. Our techniques allow to generalize these results to arbitrary dimensions. We also discuss a different approach to the problem based on reverse Hölder inequalities due to G. Chiti.
\end{abstract}
\maketitle


\section{Introduction and main results}
\subsection{Notation}

Let $(M^n,g)$ be a compact, connected $n-$dimensional Riemannian manifold with or without boundary. Let $\Delta_g:C^\infty(M) \to C^\infty(M)$ denote the negative Laplace-Beltrami operator on $M$. In local coordinates $\{x_i\}_{i=1}^n$, we write
\begin{equation} \label{local}
\Delta_g = \frac{-1}{\sqrt{det(g)}} \sum \frac{\partial}{\partial x_i} (\sqrt{det(g)} g^{ij} \frac{\partial}{\partial x_j}),
\end{equation}
where the matrix $(g^{ij})$ is the inverse matrix of $g = (g_{ij})$.

We consider the closed eigenvalue problem,
\begin{equation} \label{closed}
\Delta_g u_\lambda = \lambda u_\lambda,
\end{equation}
and when $M$ has a boundary, we impose Dirichlet eigenvalue problem,
\begin{equation} \label{Dirichlet}
 \left\{
  \begin{array}{l l}
    \Delta_g u = \lambda u  \mbox{ in } M, \\
     u=0 \mbox{ on } \partial M.\\
   \end{array} \right.
\end{equation}

In both settings, $\Delta_g$ has a discrete spectrum,
\begin{equation} \nonumber
0\leq \lambda_1(M,g) \leq \lambda_2(M,g) \leq ... \nearrow +\infty,
\end{equation}
where $\lambda_1(M,g) > 0$ if $\partial M \not= \emptyset$. Let $|| . ||_p$ be the usual $|| . ||_{L^p(M)}$ norm and let $\sigma$ be the Riemannian volume form on $M$ and let $\operatorname{Vol}_g(M)$ denote the Riemannian volume of $M$. We normalize $u$ in such a way that  $||u||_{2}^2 = 1$. If $M$ has no boundary, we require that $\int_M u d\sigma = 0$.

\subsection{Volume of superlevel sets}

We define a nodal domain $A$ of an eigenfunction $u_\lambda$ on $M$ as a maximal connected open subset of $\{u_\lambda \not=0\}$. We denote by $\mathcal{A}(u_\lambda)$ the collection of all its nodal domains.

Let us first consider the Euclidean case. It is known that nodal domains can not be too small. For instance, this can be seen by the Faber-Krahn inequality, stating that given $A_i \in \mathcal{A}(u_\lambda)$,
\begin{equation}\label{FK}
\operatorname{Vol}(A_i) \geq \left(\lambda_1(B)^{n/2} |B|\right) \lambda^{-n/2}.
\end{equation}
Denote by $V_{\delta}^i = \{ x \in A_i : |u_\lambda(x)| \geq \delta ||u_\lambda||_{L^\infty(A_i)} \}$ the $\delta$-superlevel sets of the restriction of an eigenfunction to one of its nodal domain. The next result can be seen as a refinement of that observation. Indeed, each $\delta$-superlevel set of an eigenfunction can not be too small:
\begin{lemma}\label{lemme2}
Let $n\geq 3$. For all $\delta \in (0,1)$, we have that
\begin{equation}\label{lemme2eq}
\operatorname{Vol}(V^i_{\delta})  \geq  (1 - \delta)^{\frac{n}{2}} (2(n-2))^{\frac{n}{2}} \alpha_n \lambda^{-\frac{n}{2}},
\end{equation}
where $\alpha_n$ stands for the volume of the $n$-dimensional unit ball.
\end{lemma}
\noindent The preceding lemma and its proof were suggested by F. Nazarov and M. Sodin \cite{NS}.

Letting $\delta \to 0$ in \eqref{lemme2eq} yields that $$\operatorname{Vol}(V^i_0) = \operatorname{Vol}(A_i) \geq C_{n} \lambda^{-\frac{n}{2}},$$ which is an inequality {\it à la Faber-Krahn} comparable to \eqref{FK}. However, the constant is not optimal when compared to Faber-Krahn inequality since $C_{n, \delta}$  tends to $C_{n}=(2(n-2))^{\frac{n}{2}} \alpha_n$ as $\delta \to 0$.

The proof of Lemma \ref{lemme2} is based on the maximum principle, applied to a precise linear combination of the eigenfunction $u_\lambda$ and of a certain function $w$. The function $w$ is the solution of the following Poisson problem:
\begin{equation*}
\Delta w = -\lambda \chi_{V^i_\delta}  u_{\lambda,i} \mbox{ in } \mathbb{R}^n,
\end{equation*}
where $\chi_{V^i_\delta}$ denotes the characteristic function associated to $V^i_\delta$ and $u_{\lambda,i}$ denotes the restriction of $u_\lambda$ to $A_i$. An upper bound on the function $w$ is required while applying the maximum principle. The bound is proved using decreasing rearrangement of functions, as done in \cite[p. 185]{T2}. The next result is a generalization of Lemma \ref{lemme2}, adapted to manifolds of arbitrary dimension:

\begin{theorem}\label{lemme1}
Let $\delta \in (0,1)$ and $n\geq2$. There exist $\lambda_0 >0$ and $k_{g,\delta,\lambda_0}>0$ such that for all $\lambda \geq \lambda_0$, we have that
\begin{equation}\label{lemme1eq}
\operatorname{Vol}_g (V^i_{\delta})\geq k_{g, \delta,\lambda_0}\lambda^{-\frac{n}{2}}, \quad \forall i.
\end{equation}
\end{theorem}

The proof of Theorem \ref{lemme1} is similar to the proof of its $\mathbb{R}^n$ counterpart. The key idea is to choose a specific linear combination involving $u_{\lambda,i}$ and the solution of the following Poisson problem,
\begin{equation*}
\Delta w = - \lambda \chi_{V^i_\delta}  u_{\lambda,i} \mbox{ in } M.
\end{equation*}
In order to apply the maximum principle, it is required to bound the function $w$ in terms of $\lambda$ and of the volume of $V^i_\delta$. The method used to do so differs from the one used in $\mathbb{R}^n$ since decreasing rearrangement of functions no longer works on arbitrary manifolds. Instead, we use an upper bound for Green functions on $M$ in conjunction with a certain form of the Bathtub principle (see \cite[Theorem 1.14]{LL}), that is an upper bound for the integral of a non-negative decreasing radial function:

\begin{lemma}\label{rearrangement}
Let $x_0\in M$. Let $r(x) = d_g(x_0,x)$ the Riemannian distance between $x$ and $x_0$. Let $f(r)$ denote a non-negative strictly decreasing function. Given fixed positive constant $C>0$, then
$$ \sup_{\Omega \subset X, \ \operatorname{Vol}_g(\Omega) = C} \ \ \int_\Omega f(r) d\sigma = \int_{\Omega^*} f(r) d\sigma,$$
where $\Omega^*$ is the geodesic ball centered at $x_0$ of radius $R$, where $R$ is such that $|\Omega| = |\Omega^*|$.
\end{lemma}

Lemma \ref{rearrangement} can also be seen as a weaker form of decreasing rearrangement that has the advantage of being usable in a more general setting.

\subsection{Nodal extrema on closed manifolds}

The second objective of the paper is to study the distribution of so called nodal extrema,  defined as follows:
$$m_{A_i}:=\displaystyle \max_{x \in A_i} |u_\lambda (x) |,$$
where $A_i\in\mathcal{A}(u_\lambda)$. Nodal extrema on compact surfaces were previously studied in \cite{PS}. We consider the more general case of compact Riemannian manifolds of arbitrary dimension. Since the proofs given in \cite{PS} rely on the classification of surfaces and the existence of conformal coordinates, no direct generalization of their results is possible.

Our first main result in that direction is the following:

\begin{theorem}\label{Th2}
Let $(M^n,g)$ be a compact closed manifold with $n\geq 2$. If $\lambda$ is large enough, then there exists $k_g>0$ such that
\begin{equation}\label{Th2eq2}
\sum_{i=1}^{|\mathcal{A}(u_\lambda)|} m_{A_i}^p \leq k_g \lambda^{\frac{n}{2} + p \delta(p)},
\end{equation}
holds for any $p \geq 2$. Here, $\delta(p)$ corresponds to
\begin{equation}
\text{$\delta(p) = \left\{
  \begin{array}{l l}
      \dfrac{n-1}{4} \left( \dfrac{1}{2} - \dfrac{1}{p} \right), & \quad \text{$2 \leq p \leq \dfrac{2(n+1)}{n-1},$}\\
      \dfrac{n}{2}  \left( \dfrac{1}{2} - \dfrac{1}{p} \right) - \dfrac{1}{4}, & \quad \text{$\dfrac{2(n+1)}{n-1} \leq p \leq +\infty$.}\\
   \end{array} \right.$}
\end{equation}
\end{theorem}

Note that $\delta(p)$ is C. Sogge's classical $L^p$ bounds,  $||u||_p \leq C \lambda^{\delta(p)}||u||_2$ (\cite[Ch. 5]{S}). The proof of Theorem \ref{Th2} is an application of Theorem \ref{lemme1}.

As an immediate corollary of Theorem \ref{Th2}, we have the following:

\begin{corollary}\label{CorPrinc}
Let $(M^n,g)$ be a compact closed manifold. If $\lambda$ is large enough, then there exists $k_g>0$ such that
\begin{equation}\label{Th2eq1}
\sum_{i=1}^{|\mathcal{A}(u_\lambda)|} m_{A_i} \leq k_g \lambda^{\frac{n}{2}}.
\end{equation}
\end{corollary}

Indeed, a consequence of Weyl's law and Courant's theorem is that the number of nodal domains $|\mathcal{A}(u_\lambda)|$ is bounded by $k_g\lambda^{n/2}$ (see for instance \cite{Co, Cha}). Using the latter fact and then applying Cauchy-Schwartz inequality yield that
\begin{eqnarray*}
\sum_{i=1}^{|\mathcal{A}(u_\lambda)|} m_{A_i} & \leq & \left( \sum_{i=1}^{|\mathcal{A}(u_\lambda)|} m_{A_i}^2 \cdot \sum_{i=1}^{|\mathcal{A}(u_\lambda)|} 1 \right)^{\frac{1}{2}} \\
& \leq & k_g \lambda^{n/4} |\mathcal{A}(u_\lambda)|^{\frac{1}{2}} \leq k_g \lambda^{n/2},
\end{eqnarray*}
which is the desired result.

\begin{remark}
For $p=1,2$, it is easy to see that the inequalities are sharp on $\mathbb T^n$ ($\prod \sin(nx_i), \lambda = n^2$). For $p>\dfrac{2(n+1)}{n-1}$, extremals are zonal spherical harmonics. Otherwise, the extremals are highest weight spherical harmonics.
\end{remark}

One can visualise inequalities expressed in Theorem \ref{Th2} and in Corollary \ref{CorPrinc} by considering "fine" dust particles on a vibrating membrane. Indeed, where the membrane's velocity is high, Bernoulli’s equation tells us that the air pressure is low. Since the dust particles are most influenced by air pressure, they are swept by the pressure gradient near nodal extrema (see \cite{CHRS} for some figures illustrating nodal extrema and for more information on such experiments).

\begin{remark}
One can easily obtain bounds on $m_{A_i}$ using the classical Hormander-Levitan-Avakumovic $L^\infty$ bound (see for instance \cite{S}). Indeed, it implies that there exists a constant $k_g>0$ such that $||u_\lambda||_{L^\infty(A_i)}\leq k_g \lambda^{\frac{n-1}{4}}$.  Therefore, we have that
\begin{align*}
\sum_{i=1}^{|\mathcal{A}(u_\lambda)|} ||u_\lambda||_{L^\infty(A_i)}  \leq  k_g |\mathcal{A}(u_\lambda)| \lambda^{\frac{n-1}{4}} \leq k_g \lambda^{\frac{3n-1}{4}},
\end{align*}
which is not optimal when compared to the sharp inequality given in Corollary \ref{CorPrinc}.
\end{remark}

We also obtain a generalization of \cite[Corollary 1.7]{PS}. The result is the following:

\begin{corollary}
Given $a>0$, consider nodal domains such that $m_{A_i} \geq a \lambda^{\frac{n-1}{4}}$. If $\lambda$ is large enough, then there exists $k_g >0$ such that the number of such nodal domains does not exceed $k_g a^{-\frac{2(n+1)}{n-1}}$. In particular, for fixed $a$, it remains bounded as $\lambda \to \infty$.
\end{corollary}

Indeed, letting $N_\lambda$ denote the number of such nodal domains, using $\eqref{Th2eq2}$ with $p=\frac{2(n+1)}{n-1}$, we have that
$$ N_\lambda (a\lambda^{\frac{n-1}{4}})^{\frac{2(n+1)}{n-1}} \leq \sum^{N_\lambda}_{i=1} m_{A_i}^{\frac{2(n+1)}{n-1}} \leq k_g \lambda^{\frac{n+1}{2}},$$
yielding the conclusion.


\subsection{Elliptic operators on Euclidean domains}
We obtain analogous results to Theorem \ref{Th2}. More precisely, we obtain bounds on the distribution of nodal extrema of eigenfunctions associated to the Dirichlet problem of general second order elliptic operators in the divergence form on an Euclidean bounded domain $\Omega$.

Consider the following Dirichlet eigenvalue problem:
\begin{equation}\label{prob}
 \left\{
  \begin{array}{l l}
     L(u) = \lambda u \mbox{ in } \Omega, \\
     u=0 \mbox{ on } \partial\Omega,\\
   \end{array} \right.
\end{equation}
where we consider a general elliptic operator $L$ defined as
$$ L(u) := - \displaystyle \sum_{i,j=1}^n \dfrac{\partial}{\partial x_i} (a_{ij}\dfrac{\partial u}{\partial x_j}) + c u. $$
Here, the coefficients $a_{ij}(x)$ are real measurable functions such that $a_{ij}=a_{ji}, \forall 1 \leq i,j \leq n$. We assume that $c(x)$ is a bounded measurable function such that $c(x)\geq 0$. Note that the non negativity of $c$ can be assumed without loss of generality (see \cite[Remark 1.1.3, p. 3]{H}). For convenience, we normalize the coefficients in such a way that $1$ is the lower ellipticity constant. Thus, the assumption reads
\begin{equation}\label{ellipticity}
\sum_{i,j=1}^n   a_{ij} \xi_i \xi_j \geq |\xi|^2, \forall \ \xi \in \mathbb{R}^n.
\end{equation}

We are ready to state the result:

\begin{theorem}\label{Th1}
Consider $u_\lambda$ an eigenvalue of \eqref{prob} associated to the eigenvalue $\lambda$, then
\begin{equation}\label{P1}
\sum_{i=1}^{|\mathcal{A}(u_\lambda)|} m_{A_i}  \leq K_{n,1} \operatorname{Vol}(\Omega)^{\frac{1}{2}} \lambda^{\frac{n}{2}},
\end{equation}
and
\begin{equation}\label{P2}
\sum_{i=1}^{|\mathcal{A}(u_\lambda)|} m_{A_i}^2  \leq K_{n,2}^2 \lambda^{\frac{n}{2}}.
\end{equation}
The constant $K_{n,p}$ depends on $n$ and on $p$ and is given by
\begin{equation}\label{Const}
K_{n,p}= \frac {2^{1-\frac{n}{2}} (n \alpha_n)^{\frac{-1}{p}}} { \Gamma(\frac{n}{2}) \Big( \int_0^{ j_{\frac{n}{2}-1}}  r^{p- \frac{np}{2}+n-1} J_{\frac{n}{2}-1}^p(r) dr \Big)^{\frac{1}{p}}}.
\end{equation}
\end{theorem}

The main tool to prove Theorem \ref{Th1} is Chiti's reverse Hölder inequality satisfied by any elliptic operator in divergence form with Dirichlet boundary conditions.

\begin{remark}
Since Theorem \ref{Th1} can be applied to general elliptic operators such as the Laplace-Beltrami operator in local coordinates as defined in \eqref{local}, it can also be used with a Laplacian eigenfunction on compact Riemannian manifolds provided that all its nodal domains can always be included in a single chart of $M$.
\end{remark}

\begin{remark}
A notable feature of \cite[Theorem 1.3]{PS} is that the bounds on the distribution of the nodal extrema hold for a larger class of functions defined on compact surfaces, including eigenfunctions associated to the bi-laplacian clamped plate problem. Both approaches can not be extended to the bi-laplacian case since they rely on the maximum principle, which is known not to hold for such operators.

%
%

\end{remark}

\subsection{Neumann boundary conditions in the planar case}

Let $\Omega$ be a bounded planar domain with piecewise analytic boundary. We consider the Neumann eigenvalue problem on $\Omega$, namely

\begin{equation} \label{Neumann}
 \left\{
  \begin{array}{l l}
    \Delta u =  \mu u \mbox{ in } \Omega, \\
     \frac{\partial u}{\partial n} =0 \mbox{ on } \partial \Omega.\\
   \end{array} \right.
\end{equation}

Using an argument of \cite{Polt} based on a result of \cite{TZ}, it is possible to bound the number of nodal domains touching the boundary of $\Omega$ by $C_\Omega \sqrt{\mu}$. By doing so, it is an easy matter to obtain the following:

\begin{theorem}\label{Neum}

Let $\Omega$ be a bounded planar domain with piecewise analytic boundary, then there exists $C_\Omega>0$ and $K_\Omega>0$  such that
\begin{equation}\label{eqN1}
\sum_{i=1}^{|\mathcal{A}(u_\mu)|} m_{A_i} \leq   C_\Omega  \mu,
\end{equation}
and
\begin{equation}\label{eqN2}
\sum_{i=1}^{|\mathcal{A}(u_\mu)|} m_{A_i}^2 \leq   K_\Omega \mu.
\end{equation}
\end{theorem}

\subsection{Manifolds with Dirichlet boundary conditions}

In order to obtain similar results for manifolds with boundary conditions, one has to use Sogge-Smith's adapted bounds for such setting (see \cite{SS}). For the sake of clarity, we recall these results here.

Let $(M^n,g)$ be a compact Riemannian manifold with boundary. Let $u_\lambda$ denote a Dirichlet eigenfunction associated to $\lambda$, then there exists $k_g>0$ such that
\begin{equation}
||u_\lambda||_p \leq k_g \lambda^{\frac{n}{2}(\frac{1}{2}-\frac{1}{p})-\frac{1}{4}} ||u_\lambda||_2,
\end{equation}
for $p\geq 4$ if $n \geq 4$, and $p \geq 5$ if $n=3$. One can easily adapt the proof of Theorem \ref{Th2} using Sogge-Smith results to get :

\begin{theorem}\label{Th3}
Let $(M^n,g)$ be a compact Riemannian manifold with boundary. If $\lambda$ is large enough,  there exists $k_g>0$  such that
\begin{equation}\label{Th3eq1}
\sum_{i=1}^{|\mathcal{A}(u_\lambda)|} m_{A_i} \leq k_g \lambda^{\frac{n}{2}},
\end{equation}
and
\begin{equation}\label{Th3eq2}
\sum_{i=1}^{|\mathcal{A}(u_\lambda)|} m_{A_i}^2 \leq k_g \lambda^{\frac{n}{2}}.
\end{equation}
Moreover, we have the following
\begin{equation}\label{Th3eq3}
\sum_{i=1}^{|\mathcal{A}(u_\lambda)|} m_{A_i}^p \leq k_g \lambda^{\frac{n}{2} + \frac{np}{2}(\frac{1}{2}-\frac{1}{p})-\frac{p}{4} },
\end{equation}
for any $p \geq 4$  if $n \geq 4$, and $p \geq 5$ if $n=3$.
\end{theorem}

In \cite{SS}, it is conjectured that the following bound holds:
$$||u||_p \leq C \lambda^{\alpha(p)}||u||_2,$$ where
\begin{equation}
\text{$\alpha(p) = \left\{
  \begin{array}{l l}
      \left(\frac{2}{3}+\frac{n-2}{2}\right) \left(\frac{1}{4}-\frac{1}{2p}\right), & \quad \text{$2 \leq p \leq \dfrac{6n+4}{3n-4},$}\\
     \frac{n}{2}\left(\frac{1}{2}-\frac{1}{p}\right)-\frac{1}{4}, & \quad \text{$\dfrac{6n+4}{3n-4} \leq p \leq +\infty$.}\\
   \end{array} \right.$}
\end{equation}

Hence, a version of Theorem \ref{Th3} without the restrictions could be obtained if one showed these latter bounds :

\begin{conjecture}
Let $(M^n,g)$ be a manifold with boundary. If $\lambda$ is large enough, then there exists $k_g>0$ such that
\begin{equation}\label{Th2eq3}
\sum_{i=1}^{|\mathcal{A}(u_\lambda)|} m_{A_i}^p \leq k_g \lambda^{\frac{n}{2} + p \alpha(p)},
\end{equation}
for any $p \geq 2$.
\end{conjecture}

We also obtain a generalization of \cite[Corollary 1.7]{PS} in the case of manifolds with boundary. Using \eqref{Th3eq3} with $p=\frac{6n+4}{3n-4}$, we get the following:

\begin{corollary}
Given $a>0$, consider nodal domains such that $m_{A_i} \geq a \lambda^{\frac{n-1}{4}}$. If $\lambda$ is large enough, then there exists $k_g >0$ such that the number of such nodal domains does not exceed $k_g a^{-\frac{6n+4}{3n-4}}$. In particular, for fixed $a$, it remains bounded as $\lambda \to \infty$.
\end{corollary}

%
%
%
%
%
%

\subsection{Bounds for the $p$-Laplacian}

%

For $1 < p < \infty$, the $p-$Laplacian of a function $f$ on an open bounded Euclidean domain $\Omega$ is defined by $\Delta_p f = \mbox{div}(|\nabla f|^{p-2} \nabla f).$ We consider the following eigenvalue problem:
\begin{equation}\label{pLaplacian}
  \Delta_p u + \lambda |u|^{p-2}u=0 \mbox{ in } \Omega,
\end{equation}
where we impose the Dirichlet boundary conditions. We say that $\lambda$ is an eigenvalue of $-\Delta_p$ if \eqref{pLaplacian} has a nontrivial weak solution $u_{\lambda,p} \in W^{1,p}_0(\Omega)$. That is, for any $v \in  C^\infty_0(\Omega)$,
\begin{equation}
\int_\Omega |\nabla u_\lambda|^{p-2} \nabla u_\lambda \cdot \nabla v - \lambda \int_\Omega |u_\lambda|^{p-2} u_\lambda v=0.
\end{equation}
 The function $u_\lambda$ is then called an eigenfunction of $-\Delta_p$ associated to the eigenvalue $\lambda$.  The function $u_\lambda$ is then called an eigenfunction of $-\Delta_p$ associated to $\lambda$. Note that if $p=2$, the $p$-Laplacian corresponds to the usual Laplacian and is linear. Otherwise, we say that the $p$-Laplacian is "half-linear" in the sense that it is $(p-1)$ homogeneous but not additive.

It is known that the first eigenvalue of the Dirichlet eigenvalue problem of the $p$-Laplace operator, denoted by $\lambda_{1,p}$, is characterized as,
\begin{equation}\label{var1}
\lambda_{1,p} = \min_{0 \neq u\in C^\infty_0(\Omega)} \left \{ \frac{ \int_\Omega |\nabla u|^p dx}{\int_\Omega |u|^p dx} \right \}.
\end{equation}
The infimum is attained for a function $u_{1,p} \in W^{1,p}_0(\Omega)$. In addition, $\lambda_{1,p}$ is simple and isolated. Moreover, the eigenfunction $u_{1}$  associated to $\lambda_{1,p}$ does not change sign, and it is the only such eigenfunction.

Via, for instance, the Lyusternick-Schnirelmann maximum principle, it is possible to construct $\lambda_{k,p}$ for $k\geq 2$ and hence obtain an increasing sequence of so-called variational eigenvalues of \eqref{pLaplacian} tending to $+\infty$. There exist other variational characterizations of these eigenvalues. However, no matter which variational characterization one chooses, it always remains to show that all the eigenvalues obtained that way exhaust the whole spectrum of $\Delta_p$.

Less is known about nodal geometry of eigenfunctions for the $p$-Laplace operator. For instance, it is not clear if the the interior of the set $\{ x \in \Omega : u_\lambda(x)=0\}$ is empty or not for $p$-Laplacian eigenfunctions. For more details on nodal geometry of the $p$-Laplace operator, see for instance \cite{L, P1, P2}.

Nevertheless, using a $L^\infty$ bound obtained in \cite[Lemma 4.1]{Lind}, one can still obtain an extension of \eqref{P1} for the $p$-Laplace operator.
\begin{theorem}\label{Lindq}
Let $\Omega$ be a smooth bounded open set in $\mathbb R^n$. Consider $u_{p, \lambda}$ an eigenfunction of the Dirichlet $p$-Laplacian eigenvalue problem associated to the eigenvalue $\lambda$. Let $||u_{p, \lambda}||_{p,\Omega}=1$, then we have the following:
\begin{equation}\label{LindqEQ}
 \sum_{i=1}^{|\mathcal{A}(u_\lambda)|} m_{A_i} \leq  4^n  \operatorname{Vol}(\Omega)^{1-\frac{1}{p}} \lambda^{\frac{n}{p}} .
 \end{equation}
\end{theorem}

Notice that if $p=2$, this result corresponds to what we expect in the case of the usual Laplace operator.

The Courant nodal theorem combined with the Weyl Law yield that the number of nodal domains of a Dirichlet eigenfunction associated to an elliptic operator $L$ does not exceed $C \lambda^{\frac{n}{2}}$. For the $p$-Laplacian case, the number of nodal domains $N_\lambda$ associated to an arbitrary eigenfunction is known to be bounded, see \cite{DR}. It is also shown in \cite{DR} that the number of nodal domains of an eigenfunction $u_{k}$ associated to a variational eigenvalue is bounded by $2k-2$.  Moreover, it is known that there exists two positive constants depending on $\Omega$ such that $c k^{p/n} \leq \lambda_{k,p} \leq C k^{p/n}$ (see \cite{AP}). Combining both results yields that $N_\lambda \leq C \lambda^{n/p}$ if $\lambda$ is a variational eigenvalue. We show that a similar result holds even for non-variational eigenvalue:
\begin{corollary}
For any eigenfunction of \eqref{pLaplacian} and any $a>0$, there exists a positive constant $C>0$ such that the number of nodal domains $A \in \mathcal{A}(f)$ with $m_A \geq a$ does not exceed $C a^{-1} \lambda^{\frac{n}{p}}$.
\end{corollary}

Indeed, letting $N_\lambda$ denote the number of such nodal domains, using $\eqref{LindqEQ}$, we have that
$$ N_\lambda a  \leq \sum^{N_\lambda}_{i=1} m_{A_i} \leq C \lambda^{\frac{n}{p}},$$
yielding the conclusion.

%
%

%

\subsection{Structure of the paper}

In Section \ref{Sproofs}, we prove the main results, namely we start with Lemma \ref{lemme2} in $\mathbb{R}^n$ and then we prove Theorem \ref{lemme1} for arbitrary compact Riemannian manifolds. This leads to the proof of Theorem \ref{Th2} which is an application of Theorem \ref{lemme1}. In Section \ref{S4}, we prove Theorems \ref{Th1}, \ref{Neum} and \ref{Lindq}.

\section{Proofs of main results}\label{Sproofs}

\subsection{Proof of Lemma \ref{lemme2}}

Before proving Theorem \ref{lemme1} that holds for compact Riemannian manifolds, we give a proof of such result in the Euclidean case to give the intuition behind the proof more clearly.

In order to prove Lemma \ref{lemme2}, we need a technical result concerning Poisson equation. Let $\Omega\subset\mathbb R^n, n \geq 3$, denote a bounded domain of $\mathbb R^n$ and consider the following problem:
\begin{equation} \label{poisson}
\Delta w = f \chi_\Omega \qquad \mbox{ in } \mathbb R^n,
\end{equation}
where $\chi_\Omega$ is the characteristic function of $\Omega$ and $||f(x)||_{L^\infty(\Omega)} = 1$. It is well known that the solution of such problem is given by $w(x) = (f\chi_\Omega \ast \Phi)(x)$, where $\Phi (x-y) = \frac{1}{n(n-2)\alpha_n} |x-y|^{2-n}$ is the fundamental solution of the Laplace operator.

\begin{proposition}\label{proptech1}
Let $\Omega\subset\mathbb R^n, n \geq 3$ and $||f(x)||_{L^\infty(\Omega)} = 1$. Then, we have that
$$ ||w||_{L^\infty(\Omega)} \leq \frac{1}{2(n-2)\alpha_n^{\frac{2}{n}}} \operatorname{Vol}(\Omega)^{2/n}.$$
Moreover, equality holds if $f\equiv 1$ and if $\Omega$ is a ball.
\end{proposition}

Before we give a proof, we give a quick overview of classical rearrangements of functions. Let $u$ be a measurable function defined on an open set $\Omega$. We can form the distribution function of $u$, denoted by $\mu(t)$,  the decreasing rearrangement of $u$, $u^*(s)$ into $[0,+\infty]$ and the spherically symmetric rearrangement of $u$, $u^\star$. The distribution function of u
$$\mu(t) = \operatorname{meas}\{ x \in \Omega : |u(x)| = t \} ,$$
is a right-continuous function of $t$, decreasing from $\mu(0)= |\operatorname{supp}(u)|$ to $\mu(+\infty)=0$ as $t$ increases. The decreasing rearrangement of $u$, a positive, left continuous function into $[0,+\infty]$, is defined as
$$
u^*(s) = \inf \{ t \geq 0 : \mu(t) < s \}.
$$
The spherically symmetric rearrangement of $u$ is a function $u^{\star}$ from $\mathbb{R}^n$ into $[0,+\infty]$ whose level sets $\{ x \in \mathbb{R}^n : u^{\star}(x) > t\}$ are concentric balls with the same measure as the level sets  $\{ x \in \Omega : |u(x)| > t\}$. More precisely,  $u^{\star}$  is defined as
$$
u^{\star}(x) = u^*(\alpha_n |x|^n) = \inf\{ t \geq 0 : \mu(t) < \alpha_n |x|^n \}.
$$
Note that $||u||_{\infty}=u^*(0)=u^{\star}(0)$. We refer to \cite{T} for more details on rearrangements of functions.

\begin{proof}[Proof of Proposition \ref{proptech1}]\label{Proof1}

Let us consider first the case where $f\equiv 1$ and if $\Omega$ is a ball centered at $x$ of radius $R$. Straightforward computation shows that
\begin{eqnarray*}
\left|\int_{\mathbb R^n} \chi_\Omega(y) \Phi (x-y) dy \right| &=& \frac{1}{n(n-2)\alpha_n} \int_\Omega |x-y|^{2-n} dy \\
&=& \frac{1}{(n-2)}\int_0^R r^{2-n} r^{n-1} dr \\
&=& \frac{1}{2(n-2)} R^2 = \frac{1}{2(n-2)\alpha_n^{\frac{2}{n}}} \operatorname{Vol}(B_R)^{2/n}.
\end{eqnarray*}
Now, for the general case, notice that
\begin{eqnarray*}
|w(x)| & = & \left |\int_{\mathbb R^n} f(y) \chi_\Omega(y) \Phi (x-y) dy \right| \\ & \leq & \frac{1}{n(n-2)\alpha_n}  \int_{\mathbb R^n} |f(y)| \chi_\Omega(y) |x-y|^{2-n} dy \\
&\leq & \frac{1}{n(n-2)\alpha_n}  \int_{\mathbb R^n} \chi_\Omega(y) |x-y|^{2-n} dy.
\end{eqnarray*}
The following is a  classical result of Hardy and Littlewood that can be found in \cite{HLP} :
$$\int_{\mathbb R^n} u(x) v(x) dx \leq \int_{\mathbb R^n} u^\star (x) v^\star (x) dx.$$
Therefore, since $\Phi = \Phi^{\star}$, we get that
\begin{eqnarray*}
|w(x)| & \leq &  \int_{\mathbb R^n} \chi_\Omega(y) \Phi(x-y) dy \\
& \leq &  \int_{\mathbb R^n} \chi_{\Omega^\star}(y)\Phi^{\star}(x-y) dy \\
& = & \frac{1}{n(n-2)\alpha_n}  \int_{\Omega^\star}  |x-y|^{2-n} dy,
\end{eqnarray*}
where $\Omega^\star$ denotes a ball centered at $x$ of same volume of $\Omega$. By the previous case, one gets the desired result.
\end{proof}

\begin{remark}
The last step of Proof \ref{Proof1} is to show that
 \begin{equation}\label{bt}
 \int_{\Omega} \Phi(x-y)dy \leq \int_{\Omega^\star} \Phi(x-y)dy.
 \end{equation}
A generalization of \eqref{bt} is given by Lemma \ref{rearrangement}.
\end{remark}

That being done, we can start the main proof of this section.

\begin{proof}[Proof of Lemma \ref{lemme2}]
Renormalize $u_{\lambda}$ such that $||u_{\lambda}||_\infty=1$. Consider $\delta \in (0,1)$. We want to show that there exists a constant $C_{n,\delta}>0$ such that
$$\operatorname{Vol}(V^i_{\delta})\geq C_{n,\delta} \lambda^{-\frac{n}{2}}.$$
Let $g = u - \delta $. We have that $\Delta g = \Delta u_{\lambda,i}  =  \lambda u_{\lambda,i} $ in $V^i_{\delta}$. By Proposition \ref{proptech1}, there exists $w(x)$ satisfying \eqref{poisson} with $f = -\lambda u_{\lambda,i} $ and $\Omega = V^i_{\delta}$ such that $||w||_\infty \leq \frac{1}{2(n-2)\alpha_n^{\frac{2}{n}}} \lambda \operatorname{Vol}(V^i_{\delta})^{\frac{2}{n}}$. Consider the function $g+w$ on $V^i_{\delta}$. On the boundary, we have that $g+w \leq \frac{1}{2(n-2)\alpha_n^{\frac{2}{n}}} \lambda \operatorname{Vol}(V^i_{\delta})^{\frac{2}{n}}$. Consider $x_0$ in $V^i_{\delta}$ such that $u_{\lambda,i}(x_0)=1 = ||u_\lambda||_\infty$. Thus, we have that $(g+w)(x_0) \geq (1 - \delta) - \frac{1}{2(n-2)\alpha_n^{\frac{2}{n}}} \lambda \operatorname{Vol}(V^i_{\delta})^{\frac{2}{n}}$.

Moreover, since $\Delta (g+w) = \lambda u_{\lambda,i}  - \lambda u_{\lambda,i} = 0$, we can use the maximum principle on $g+w$. This implies that
\begin{eqnarray*}  (1 - \delta) - \frac{1}{2(n-2)\alpha_n^{\frac{2}{n}}} \lambda \operatorname{Vol}(V^i_{\delta})^{\frac{2}{n}} & \leq & \frac{1}{2(n-2)\alpha_n^{\frac{2}{n}}} \lambda \operatorname{Vol}(V^i_{\delta})^{\frac{2}{n}} \\ \iff  \operatorname{Vol}(V^i_{\delta})^{\frac{2}{n}} &\geq & \frac{1}{2} (1 - \delta)\left(\frac{\lambda}{2(n-2)\alpha_n^{\frac{2}{n}}}\right)^{-1}, \end{eqnarray*}
yielding that $\operatorname{Vol}(V^i_{\delta}) \geq (1 - \delta)^{\frac{n}{2}} (2(n-2))^{\frac{n}{2}} \alpha_n \lambda^{-\frac{n}{2}} $.
\end{proof}

\subsection{Proof of Theorem \ref{lemme1} }

The proof of Theorem \ref{lemme1} for manifolds keeps the same spirit than the proof for $\mathbb{R}^n$. The main difference with the previous proof is that we can not use the Proposition \ref{proptech1} since it relies on the fundamental solution of the Laplace operator on $\mathbb{R}^n$. We shall instead consider the Green representation of solution to poisson problem on $M$.

Let $\Omega$ be a compact smooth domain of $(M^n,g)$ where $n\geq 3$. It is known that there exists a Green function (see for instance \cite{SY}), namely a smooth function $G$ defined on $\Omega \times \Omega \setminus \{ (x,x) : x \in \Omega \}$ such that
\begin{itemize}
\item $ G(x,y) = G(y,x), x\neq y$;
\item For fixed $y$, $ \Delta_x G(x,y) = 0, \forall x \neq y $;
\item $ G(x,y) \geq 0 $ and $G$ vanishes on the boundary for $\Omega$;
\item As $x\to y$ for fixed $y$, $ G(x,y) \leq  \rho(x,y)^{2-n} (1 + o(1)), n \geq 3$, where $\rho(x,y)$ is the geodesic distance between $x$ and $y$ (see \cite[p. 81]{SY}).
\end{itemize}
Moreover, if we consider the following problem:
\begin{equation} \label{Poisson}
 \Delta_g w = f \mbox{ in } M;
\end{equation}
then, it is known that $w$ is given by
$$ w(y) = \int_M G(x,y) f(x) d\sigma. $$

\begin{proposition}\label{proptech2}
Let $n\geq 3$, $||u_\lambda||_\infty = 1$ and $\delta \in (0,1)$. Let $A_i$ denote a nodal domain of $u_\lambda$ and $V_{\delta}^i = \{ x \in A_i : |u_\lambda(x)| \geq \delta m_{A_i} \}$. There exist $\lambda_0$ and $k_{g,\lambda_0}>0$ such that $\forall \lambda > \lambda_0$ and for any $x_0 \in V^i_\delta$, we have that
$$ | w(x_0) | \leq k_{g,\lambda_0} \lambda \operatorname{Vol}_g(V^i_{\delta})^{\frac{2}{n}}.$$
\end{proposition}

We want to prove an analogous result to Proposition \ref{proptech1}. To do so, we treat split the argument into two cases depending on if the volume of $V^i_\delta$ is "large" or "small". We define "small $V^i_\delta$" in such a way that we can apply normal coordinates. This becomes handy since Green functions on $M$ behaves roughly like the fundamental solution of the Laplace operator on $\mathbb{R}^n$. Using the Lemma \ref{rearrangement}, it is then possible to bound $w$ like claimed.

\begin{proof}[Proof of Proposition \ref{proptech2}]

Let $A_i$ a nodal domain of $u_\lambda$ and let $x_0$ be any point such that $u_\lambda(x_0) = m_{A_i}$.

Let $B_{x_0}(r):= \operatorname{exp}_{x_0} (B_0(r))$ denote the geodesic ball of radius $r$ centered at $x$. It is known that for $r$ small enough, we have that
$$\operatorname{Vol}_g (B_{x_0}(r)) = r^n \operatorname{Vol}(B_0(1)) \left( 1 - \frac{ \operatorname{scal}_g(x_0)) }{ 6 ( n + 2) } r^2 + o(r^2) \right), $$
where $\operatorname{scal}_g(x_0)$ denotes the scalar curvature at $x_0$. Therefore, there exists $\epsilon\in(0,1)$ such that for all $0 < r \leq \epsilon \leq \operatorname{injrad} (M,g) $, there exist $A_{g}>0$ and $B_{g}>0$ such that
\begin{equation}\label{boundsvol}
A_{g} r^n \leq \operatorname{Vol}_g (B_{x_0}(r)) \leq B_{g} r^n. \end{equation}

Renormalize $u_{\lambda}$ such that $||u_{\lambda}||_\infty=1$. Fix a nodal domain $A_i$ and $x_0 \in A_i$.

Let $\lambda_0 = B_g^{-2/n} \epsilon^{-2}$. Notice that if $\lambda \geq \lambda_0$ and if $\operatorname{Vol}_g(V^i_\delta) > \operatorname{Vol}_g(B_{x_0}(\epsilon))$, the result holds with $k_g = \frac{A_g}{B_g}$.

On the other hand, if $\lambda \geq \lambda_0$, but  $\operatorname{Vol}_g(V^i_\delta) \leq \operatorname{Vol}_g(B_{x_0}(\epsilon))$, it is always possible to pick $R$ such that $\operatorname{Vol}_g (V_{\delta}^i) = \operatorname{Vol}_g (B_R(x_0))$ and $R \leq \epsilon$ hold.

Let us now consider the auxiliary problem defined by \eqref{Poisson} with $f = -\chi_{V^i_\delta} \lambda u_{\lambda,i}$. By definition of a Green function, we have that
\begin{eqnarray*}\label{estim2}
|w(x_0)| & = & \left| \lambda \int_{V^i_{\delta}} G(x,x_0) u_\lambda(x) d\sigma \right| \\
& \leq & \lambda \int_{V^i_{\delta}}  G(x,x_0)  d\sigma .
\end{eqnarray*}
Using upper bounds on the Green function (see bounds proved in \cite{R}), we have that there exists $C_g > 0$ such that
$$G(x,x_0) \leq C_g \rho(x,x_0)^{2-n}, \ \forall x \neq x_0, $$
implying that
$$ |w(x_0)| \leq C_g \lambda  \int_{V^j_{\delta}} \rho(x,x_0)^{2-n} d\sigma .$$

As it was done in $\mathbb{R}^n$, we need to integrate on a ball to obtain a straightforward computable integral. To do so, we use Lemma \ref{rearrangement} whose proof can be found in Section \ref{Srearrang}. Applying Lemma \ref{rearrangement}, we get the following:
 $$C_g \lambda  \int_{V^j_{\delta}} \rho(x,x_0)^{2-n} d\sigma   \leq  C_g \lambda  \int_{(V^i_\delta)^*} \rho^{2-n} d\sigma ,$$
 where $(V^i_\delta)^* = B_{x_0}(R)=\operatorname{exp}_{x_0} (B_{0}(R))$.

Using Gauss's Lemma, we now have that
\begin{eqnarray*}\label{estim4}
|w(x_0)| & \leq & C_g \lambda \int_{(V^i_\delta)^*} \rho^{2-n} (1 - \frac{1}{6} R_{kl}x^k x^l + O(|x|^3)) dx^1 dx^2 \ldots dx^n  \\
& \leq & C_g \lambda \left(\dfrac{n \omega_n}{2} R^2 - \dfrac{n \omega_n Scal_g(x_0)}{6} \dfrac{R^4}{4} + O(R^5)\right) \\
& \leq & C_g \lambda \dfrac{n \omega_n}{2} R^2 \left( 1 - \dfrac{ Scal_g(x_0)}{6} \dfrac{R^2}{2} + O(R^3)\right) \\
& \leq & C_{g} B_g E_g \lambda \operatorname{Vol}_g(B_{x_0}(R))^{\frac{2}{n}} = k_{g,\lambda_0} \lambda \operatorname{Vol}_g(V_{\delta}^i)^{\frac{2}{n}}.
\end{eqnarray*}




\end{proof}

The last step to prove Theorem \ref{lemme1} is very similar to the last step in the proof of Lemma \ref{lemme2}.

\begin{proof}[Proof of Theorem \ref{lemme1}]
Renormalize $u_{\lambda}$ such that $||u_{\lambda}||_\infty=1$. Let $g = u - \delta + w$. On the boundary of $V^i_{\delta}$, we have that $g= \delta - \delta = 0$. Consider any $x_0$ in $V^i_{\delta}$ such that $u_{\lambda,i}(x_0)=1$. By Proposition \ref{proptech2}, we have that $g(x_0) \geq (1 - \delta) - C_{g,\lambda_0} \lambda \operatorname{Vol}(V^i_{\delta})^{\frac{2}{n}}$.

Moreover, since $\Delta g = \Delta u_{\lambda,i}  + \Delta w =  \lambda u_{\lambda,i} - \lambda u_{\lambda,i} = 0 $ in $V^i_{\delta}$, we can use the maximum principle on $g$. This implies that
$$  (1 - \delta)  - C_{g,\lambda_0} \lambda \operatorname{Vol}_g(V^i_{\delta})^{\frac{2}{n}} \leq 0 \iff  \operatorname{Vol}_g(V^i_{\delta}) \geq  k_{g,\lambda_0}(1 - \delta)^{\frac{n}{2}} \lambda^{-\frac{n}{2}}. $$
\end{proof}

\begin{remark}
Note that for compact surfaces, Theorem \ref{lemme1} is implied by a bound on the inner radius of nodal domains, namely \cite[Lemma 10]{M}. Indeed, the latter lemma implies that there exists a ball of radius $\frac{\epsilon}{\sqrt{\lambda}}$ that is centered at a "nodal extrema", thus implying the result.
\end{remark}

\subsection{Proof of Theorem \ref{Th2} }

Let $\delta \in (0,1)$ and $\lambda$ be large enough. Recall that $\mathcal{A}(u)= \{ A_i \}_{i=1}^{|\mathcal{A}(u_\lambda)|}$ is the collection of the nodal domains of $u_\lambda$. Consider
\begin{equation}
u_\lambda = \sum^{|\mathcal{A}(u_\lambda)|}_{i=1} u_{\lambda,i}  \quad \text{where $u_{\lambda,i} = \left\{
  \begin{array}{l l}
     u_\lambda & \quad \text{if $x\in\mathcal{A}_i,$}\\
     0 & \quad \text{elsewhere.}\\
   \end{array} \right.$}
\end{equation}
Observe that $\lambda=\lambda_1(A_i)$ since $u_{\lambda,i}$ does not vanish in $A_i$ (see \cite{Cha} or \cite{H}). Apply Theorem \ref{lemme1} in order to get the following:

\begin{align*}
\int_{A_i} |u_{\lambda,i}|^p & d\sigma \geq \int_{V^i_\delta} \dfrac{m_{A_i}}{2}^p d\sigma  \\
& =  \dfrac{m_{A_i}}{2}^p  \operatorname{Vol}_g(V^i_\delta) \\
& \geq k_{g,\delta,\lambda_0} m_{A_i}^p \lambda^{-\frac{n}{2}}.
\end{align*}
If we sum over all nodal domains, we get that
\begin{equation}\label{alex}
\int_M |u_\lambda|^p d\sigma = \sum_{i=1}^{|\mathcal{A}(u_\lambda)|} \int_{A_i} |u_{\lambda,i}|^p d\sigma \geq k_{g,\delta,\lambda_0} \lambda^{-\frac{n}{2}} \sum_{i=1}^{|\mathcal{A}(u_\lambda)|} m_{A_i}^p.
\end{equation}

To obtain \eqref{Th2eq2}, simply use Sogge's $L^p$ bounds $||u_\lambda||_p \leq \lambda^{\delta(p)} ||u_\lambda||_2$ in \eqref{alex}.

Notice that one can read off \eqref{Th2eq1} using the latter argument. Indeed, since
\begin{equation}\nonumber
 \int_M |u_\lambda| d\sigma  \leq  \operatorname{Vol}_g(M)^{\frac{1}{2}} \bigg( \int_M |u_\lambda|^2 d\sigma  \bigg)^{\frac{1}{2}} = \operatorname{Vol}_g(M)^{\frac{1}{2}},
\end{equation}
if we take $p=1$ in \eqref{alex}, we get
\begin{equation}\nonumber
\operatorname{Vol}_g(M)^{\frac{1}{2}} \geq \int_M |u_\lambda| d\sigma  \geq k_{g,\delta,\lambda_0} \lambda^{-\frac{n}{2}} \sum_{i=1}^{|\mathcal{A}(u_\lambda)|} m_{A_i},
\end{equation}
yielding \eqref{Th2eq1}.

\subsection{Proof of Lemma \ref{rearrangement}}\label{Srearrang}

The following is a more general statement of Lemma \ref{rearrangement}, which we prove thereafter.
\begin{lemma}
Let $(X,d, \Sigma, \mu)$ denote a sigma-finite metric measure space and fix $x_0\in X$. Let $r(x) = d(x_0,x)$. Let $f(r)$ denote a non-negative strictly decreasing function. Let $C$ be a fixed positive constant, then
$$ \sup_{\Omega \subset X, \ \mu(\Omega) = C} \ \ \int_\Omega f(r) d\mu = \int_{\Omega^*} f(r) d\mu,$$
where $\Omega^*$ is the ball centered at $x_0$ of radius $R$, where $R$ is such that $|\Omega| = |\Omega^*|$.
\end{lemma}
This result is an application of the Bathtub principle. The following is a reformulation of \cite[Theorem 1.14]{LL} adapted to our setting:
\begin{theorem}
Let $f$ be a real-valued measurable strictly decreasing function on a sigma finite measure space $(X, \Sigma, \mu)$. Fix $G = \mu(\Omega)$ and consider the class of functions $$\mathcal{C} = \left\{ 0 \leq g \leq 1 : \int_X g d\mu= G\right\}.$$
Let $E\subset X$, then
$$I=\sup_{g\in\mathcal{C}} \int_E fg d\mu$$
is obtained by taking $g = \chi_{\{f \geq s\}},$ where $s$ is such that $\mu(\{x:f(x) \geq s\}) = G$.
\end{theorem}

Notice that the Bathtub principle tells us that the supremum is given by $g = \chi_{\{f \geq s\}},$ which is the characteristic function of a ball since $f$ is a strictly decreasing function of $r$, yielding the desired result.

\section{Proof of Theorems \ref{Th1}, \ref{Neumann}, and \ref{Lindq}}\label{S4}

\subsection{Proof of Theorem \ref{Th1}}

We present the background required to obtain Theorem \ref{Th1}. For any fixed positive $\lambda$, we consider the $n-$ball,
\begin{equation}\label{ball}
B_\lambda^n = \{ x\in\mathbb{R}^n : |x| \leq j_{\frac{n}{2}-1} \lambda^{-\frac{1}{2}} \},
\end{equation}
where $ j_{\frac{n}{2}-1}$ is the first positive zero of the Bessel function $J_{\frac{n}{2}-1}$.
It is easy to see that the following problem,
\begin{equation*}
 \left\{
  \begin{array}{l l}
    \Delta z= \mu z \mbox{ in } B_\lambda^n,  \\
     z=0 \mbox{ on } \partial B_\lambda^n, \\
   \end{array} \right.
\end{equation*}
has its first eigenvalue equal to $\lambda$, and that the corresponding eigenfunction is given by
\begin{equation}\label{defZ}
z(x) = |x|^{1-\frac{n}{2}} J_{\frac{n}{2}-1}(\lambda^{\frac{1}{2}}|x|).
\end{equation}

We use the following result, due to G. Chiti (see \cite{Chi1,Chi2}), in the proof:

\begin{proposition}[{\cite[Theorem 2]{Chi1}}]\label{ThChiti}
Let $u$ be a function satisfying \eqref{prob} and consider $z(x)$, the eigenfunction to the Dirichlet eigenvalue problem on $B^n_\lambda$ defined above. Then, for any $p\geq 1$,
\begin{equation}\label{Chiti1}
||u||_{\infty}\Bigg(\int_\Omega |u|^p\Bigg)^{-\frac{1}{p}} \leq ||z||_{\infty} \Bigg(\int_\Omega z^p\Bigg)^{-\frac{1}{p}},
\end{equation}
with equality if and only if $\Omega$ is a ball, $c=0$, $a_{ij}=\delta_{ij}$, $\lambda$ is equal to the first eigenvalue of the equality in \eqref{prob} and $ |\Omega| = |B^n_\lambda|$, where $|E|$ denotes the Lebesgue measure of the set $E$.
\end{proposition}

Remark that we can compute the right hand side of \eqref{Chiti1} to obtain the following isoperimetric inequality,
\begin{equation}\label{INEQ1}
||u||_{\infty} \leq  K_{n,p}\lambda^{\frac{n}{2p}} ||u||_{p},
\end{equation}
where $K_{n,p}$ is the constant defined in \eqref{Const}. Indeed, start by computing $||z||_{\infty}$. The fact that $r^{\frac{n}{2}-1}J_{\frac{n}{2}-1}(r)$ attains its maximum at $r=0$ follows from Poisson's integral (see \cite[Section 3.3]{W}). Thus, we have that
\begin{align}\label{equa1}
z(0) &= \lim_{|x|\to 0} \frac{ J_{\frac{n}{2}-1}(\lambda^{\frac{1}{2}}|x|)}{|x|^{\frac{n}{2}-1}} = \frac{\lambda^{\frac{n}{4}-\frac{1}{2}}}{2^{\frac{n}{2}-1}\Gamma(\frac{n}{2})}.
\end{align}
Since $z(x)$ is a radial function, we get that
\begin{equation} \label{equa2}
\Big(\int_\Omega z^p\Big)^{-\frac{1}{p}}  = (n C_n)^{\frac{-1}{p}} \lambda^{\frac{1}{2}-\frac{n}{4}+\frac{n}{2p}}  \Big( \int_0^{ j_{\frac{n}{2}-1}}  r^{p- \frac{np}{2}+n-1} J_{\frac{n}{2}-1}^p(r) dr \Big) ^{\frac{-1}{p}}.
\end{equation}
Combine \eqref{equa1} and \eqref{equa2}, and plug them into \eqref{Chiti1} to get \eqref{INEQ1}.

\begin{proof}[Proof of Theorem \ref{Th1}]

We start by obtaining \eqref{P2}. Let us decompose $u_\lambda$ the following way,
\begin{equation}
u_\lambda = \sum^{|\mathcal{A}(u_\lambda)|}_{i=1} u_i  \quad \text{where $u_i = \left\{
  \begin{array}{l l}
     u_\lambda & \quad \text{if $x\in\mathcal{A}_i$}\\
     0 & \quad \text{elsewhere.}\\
   \end{array} \right.$}
\end{equation}

 Since $\operatorname{supp}(u_i)\cap \operatorname{supp}(u_j) = \varnothing $ for $ i \neq j$, we note that
\begin{equation} \label{Cnorm}
1 = ||u_\lambda||^2_{L^2(M)}
			 = \int_\Omega  \sum^{|\mathcal{A}(u_\lambda)|}_{i=1} u_i^2 \\
			 = \sum^{|\mathcal{A}(u_\lambda)|}_{i=1} \int_{A_i} u_i^2
			= \sum^{|\mathcal{A}(u_\lambda)|}_{i=1} ||u_i||^2_{L^2(A_i)}.
\end{equation}

 Recall that each $u_i$ corresponds to an eigenfunction of the Dirichlet problem on these nodal domains. Indeed, since $u_i$ does not vanish in $A_i$, it corresponds to the first eigenfunction on $\mathcal{A}_i$ and $\lambda_1(\mathcal{A}_i) = \lambda$ by a corollary of Courant's theorem (see \cite{H}).

Thus, we can apply \eqref{INEQ1} with $p=2$ to each $u_i$ so that for all $1 \leq i \leq |\mathcal{A}(u_\lambda)|$, we obtain that
\begin{equation*}
||u_i||_{L^\infty(\mathcal{A}_i)} \leq K_{n,2} \lambda^{\frac{n}{4}} ||u_i||_{L^2(\mathcal{A}_i)}.
\end{equation*}
Therefore, we get that
\begin{equation*}
m_{A_i} =  \sup_{x\in\mathcal{A}_i} |u_i(x)| \leq K_{n,2}\lambda^{\frac{n}{4}}||u_i||_{L^2(\mathcal{A}_i)}.
\end{equation*}
Squaring each side and summing over all nodal domains yield that
\begin{equation*}
\sum^{|\mathcal{A}(u_\lambda)|}_{i=1} m_{\mathcal{A}_i}^2 \leq K_{n,2}^2 \lambda^{\frac{n}{2}} \sum^{|\mathcal{A}(u_\lambda)|}_{i=1}||u_i||^2_{L^2(\mathcal{A}_i)},
\end{equation*}
and we obtain \eqref{P2} by applying \eqref{Cnorm} to the latter equation. In order to get \eqref{P1}, we use \eqref{INEQ1} with $p=1$, to get
\begin{equation*}
||u_i||_{L^{ \infty}(\mathcal{A}_i)} \leq K_{n,1} \lambda^{\frac{n}{2}} ||u_i||_{L^1(\mathcal{A}_i)}.
\end{equation*}
If we sum over all nodal domains and keep in mind that $\operatorname{supp}(u_i)\cap \operatorname{supp}(u_j) = \varnothing $ for $ i \neq j$ , we then get
\begin{align*}
\sum^{|\mathcal{A}(u_\lambda)|}_{i=1} m_{\mathcal{A}_i} & \leq K_{n,1}\lambda^{\frac{n}{2}} \sum^{|\mathcal{A}(u_\lambda)|}_{i=1}||u_i||_{L^1(\mathcal{A}_i)} \\
&  = K_{n,1}\lambda^{\frac{n}{2}} ||u_{\lambda}||_{L^1(\Omega)} \\
& \leq  K_{n,1}\lambda^{\frac{n}{2}} ||u_{\lambda}||_{L^2(\Omega)}  \operatorname{Vol}(\Omega)^{\frac{1}{2}}.
\end{align*}
The last line follows from Cauchy-Schwartz inequality. Since $||u_{\lambda}||_{L^2(\Omega)}=1$, the proof is completed.
\end{proof}

\subsection{Proof of Theorem \ref{Neumann}}

Let $I_1$ denote the family of indexes of nodal domains touching the boundary of $\Omega$ and let $I_2 = |\mathcal{A}(u_\mu)| \setminus I_1$. Let us start by obtaining \eqref{eqN2}

Notice that nodal domains whose index is in $I_2$ are such that the eigenfunction $u$ restricted to them corresponds to the first eigenfunction of the Dirichlet eigenvalue problem on such $A_i$, so that $\mu = \lambda_1(A_i)$. Therefore, it is possible to use \eqref{INEQ1} with $p=2$ as done in the proof of Theorem \ref{Th1} in order to get that
$$ \sum_{i \in I_2} m_{A_i}^2  \leq C \mu.$$
As for nodal domains whose index is in $I_1$, since by the Hormander-Levitan-Avakumovic $L^\infty$ bound, we have that $m_{A_i} \leq C \mu^{1/4}$, we get that
$$ \sum_{i \in I_1} m_{A_i}^2 \leq C \sqrt{\mu} \cdot (\mu^{1/4})^2 = C \mu,$$
yielding \eqref{eqN2}.

The same reasoning can be applied to obtain \eqref{eqN1}, namely
$$  \sum_{i \in I_1} m_{A_i} + \sum_{i \in I_2} m_{A_i} \leq C \sqrt{\mu} \cdot \mu^{1/4} + C \mu \leq C' \mu,$$
yielding \eqref{eqN1}.

\subsection{Proof of Theorem \ref{Lindq}}

The proof is based on the following result :
\begin{lemma}[Lemma 4.1 in \cite{Lind}]\label{LL41}
Let $u_{p,1}$ denote the first eigenfunction of the Dirichlet $p$-Laplacian eigenvalue problem on a bounded Euclidean domain $\Omega \subset \mathbb R^n$, then
$$||u_{p,1}||_{L^\infty(\Omega)} \leq 4^n \lambda^{\frac{n}{p}} ||u_{p,1}||_{L^1(\Omega)}. $$
Note that the constant term $4^n$ is not sharp.
\end{lemma}

\begin{remark}
One difference between Chiti-type inequalities and the preceding lemma is  that Chiti-type inequalities apply to any eigenfunction of the Dirichlet eigenvalue problem rather than only to the first one. However, the generalization of Chiti's results to the $p$-Laplace operator (see \cite{AFT}) is of the form
$$ ||u||_r \leq K(r,q,p,n, \lambda) ||u||_q,$$
where $u$ is any eigenfunction associated to eigenvalue $\lambda$, $0 < q < r \leq +\infty$. It is important to notice that the constant $K(r,q,p,n, \lambda)$ is not explicit (since we can not compute the eigenfunctions of the ball explicitly). Thus, we cannot use it as it was done for the Laplace operator.
\end{remark}

We are ready to prove Theorem \ref{Lindq}.
\begin{proof}
Let $||u_{p,\lambda}||_p = 1$. Consider $A_i \subset \Omega$ a nodal domain of $u_{p,\lambda}$. Let us decompose $u_{p, \lambda}$ the following way,
\begin{equation}
u_{p,\lambda}= \sum^{|\mathcal{A}(u_\lambda)|}_{i=1} u_i  \quad \text{where $u_i = \left\{
  \begin{array}{l l}
     u_{p,\lambda} & \quad \text{if $x\in\mathcal{A}_i$}\\
     0 & \quad \text{elsewhere.}\\
   \end{array} \right.$}
\end{equation}
Since $u_{i}$ corresponds to the first eigenfunction of the Dirichlet $p$-Laplacian eigenvalue problem on $A_i$,   Lemma \ref{LL41} yields that
$$ ||u_i||_{\infty, A_i} \leq 4^n \lambda^{\frac{n}{p}} ||u_i||_{1, A_i}, \qquad \forall \quad 1\leq i \leq |\mathcal{A}(u_\lambda)|.$$
Therefore, after summing over all nodal domains, we get that
\begin{eqnarray*}
\sum_{i=1}^{|\mathcal{A}(u_\lambda)|} ||u_i||_{L^\infty(A_i)} = \sum_{i=1}^{|\mathcal{A}(u_\lambda)|} m_{A_i} &\leq& 4^n \lambda^{\frac{n}{p}} \sum_{i=1}^{|\mathcal{A}(u_\lambda)|} ||u_i||_{L^1(A_i)} \\
&\leq& 4^n  \lambda^{\frac{n}{p}} ||u_{p, \lambda}||_{L^1(\Omega)} \\
&\leq& 4^n \operatorname{Vol}(\Omega)^{1-\frac{1}{p}} \lambda^{\frac{n}{p}} ||u_{p, \lambda}||_{L^p(\Omega)}\\
& =  & 4^n \operatorname{Vol}(\Omega)^{1-\frac{1}{p}} \lambda^{\frac{n}{p}}.
\end{eqnarray*}
\end{proof}

\subsection*{Acknowledgments}

 This research is part of my Ph.D. thesis at the Universit\'e de Montreal under the supervision of Iosif Polterovich. The author is grateful to Fedor Nazarov and Mikhail Sodin for communicating the proof of Lemma \ref{lemme2}. The author is indebted to Almut Burchard, for pointing out that Lemma \ref{rearrangement} is in fact an application of the Bathtub principle. The author would also like to thank Laurence Boulanger, Yaiza Canzani, Marlène Frigon, Benoit Kloeckner, Dan Mangoubi, Frédéric Robert, Leonid Polterovich, and Mikhail Sodin for useful discussions.

\vspace{3 mm}

{\scshape Département de mathématiques et de statistique,
Université de Montréal, CP 6128 succ. Centre-Ville, Montréal,
H3C 3J7, Canada.}

\emph{E-mail address:}   \verb"gpoliquin@dms.umontreal.ca"

\end{document}